\documentclass[a4paper,11pt]{article}
\usepackage{braket,hyperref}
\usepackage{amssymb,mathtools,amsthm,amsmath}
\usepackage{tikz}
\usepackage{xargs}
\usepackage{enumerate}

\usepackage[colorinlistoftodos,prependcaption,textsize=small]{todonotes}
\usepackage{verbatim}

\theoremstyle{plain}
\newtheorem{theorem}{\bf Theorem}

\newtheorem{conjecture}[theorem]{\bf Conjecture}

\newtheorem{corollary}[theorem]{\bf Corollary}
\newtheorem{lemma}[theorem]{\bf Lemma}
\theoremstyle{definition}

\newenvironment{remark}[1][Remark.]{\begin{trivlist}
		\item[\hskip \labelsep {\bfseries #1}]}{\end{trivlist}}

\numberwithin{theorem}{section}
\numberwithin{equation}{section}

\newcommand{\Rea}{{\mathbb R}}

\newcommand{\zhomology}[2]{\tilde{H}_{#1}\left(#2\right)}

\DeclareMathOperator{\lk}{lk}

\DeclareMathOperator{\cost}{cost}

\DeclareMathOperator{\conv}{\text{conv}}

\usepackage{authblk}

\begin{document}

    \title{Extensions of the Colorful Helly Theorem for $d$-collapsible and $d$-Leray complexes}

\author[1]{Minki Kim\thanks{\href{mailto:minkikim@gist.ac.kr}{minkikim@gist.ac.kr}. Minki Kim was supported by the Basic Science Research Program through the National Research Foundation of Korea (NRF) funded by the Ministry of Education (NRF-2022R1F1A1063424)}}
 \affil[1]{Division of Liberal Arts and Sciences, Gwangju Institute of Science and Technology (GIST), Gwangju, Republic of Korea}
 
    \author[2]{Alan Lew\thanks{\href{mailto:alanlew@andrew.cmu.edu}{alanlew@andrew.cmu.edu}.}}
    \affil[2]{Dept. Math. Sciences, Carnegie Mellon University, Pittsburgh, PA 15213, USA}

	\date{}
	\maketitle

\begin{abstract}
    We present extensions of the Colorful Helly Theorem for $d$-collapsible and $d$-Leray complexes, providing a common generalization to the matroidal versions of the theorem due to Kalai and Meshulam, the ``very colorful" Helly theorem introduced by Arocha, B\'ar\'any, Bracho, Fabila and Montejano, and the ``semi-intersecting" colorful Helly theorem proved by Montejano and Karasev.

    As an application, we obtain the following extension of Tverberg's Theorem: Let $A$ be a finite set of points in $\Rea^d$ with $|A|>(r-1)(d+1)$. Then, there exist a partition $A_1,\ldots,A_r$ of $A$ and a subset $B\subset A$ of size $(r-1)(d+1)$, such that $\cap_{i=1}^r \conv( (B\cup\{p\})\cap A_i)\neq\emptyset$ for all $p\in A\setminus B$. That is, we obtain a partition of $A$ into $r$ parts that remains a Tverberg partition even after removing all but one arbitrary point from $A\setminus B$.
\end{abstract}

\section{Introduction}

Let $\mathcal{F}$ be a family of (not necessarily distinct) sets, colored with $r$ colors. Formally, we have a partition $\mathcal{F}=\mathcal{F}_1\cup \cdots\cup \mathcal{F}_r$. We call the subfamilies $\mathcal{F}_i$ the \emph{color classes} of $\mathcal{F}$.
We say that a subfamily $\mathcal{F}'\subset\mathcal{F}$ is \emph{colorful} if it contains at least one set from each color class. We say that $\mathcal{F}'$ is \emph{intersecting} if $\cap_{F\in \mathcal{F}'} F\neq \emptyset$. 
The Colorful Helly Theorem, observed by Lov\'asz, asserts the following:

\begin{theorem}[Lov\'asz's Colorful Helly Theorem (see \cite{barany1982generalization})]
\label{thm:cht}
Let $\mathcal{C}$ be a finite family of convex sets in $\Rea^d$ colored with $d+1$ colors. If all colorful subfamilies of $\mathcal{F}$ of size $d+1$ are intersecting, then one of the color classes is intersecting.
\end{theorem}

Many combinatorial properties of families of convex sets, and in particular Helly-type results like Theorem \ref{thm:cht}, can be described in terms of the nerve of the family. Given a finite family $\mathcal{F}$ of sets, the \emph{nerve} of $\mathcal{F}$ is the simplicial complex on vertex set $\mathcal{F}$ whose simplices are the subfamilies $\mathcal{F}'\subset\mathcal{F}$ that are intersecting. We say that a simplicial complex $X$ is \emph{$d$-representable} if it is isomorphic to the nerve of some family of convex sets in $\Rea^d$.

Two related notions are those of $d$-collapsibility and $d$-Lerayness, introduced by Wegner in \cite{Weg75}. Let $X$ be a simplicial complex. A simplex $\sigma\in X$ is called a \emph{free face} of $X$ if $\sigma$ is contained in a unique maximal face $\tau$ (possibly equal to $\sigma$) of $X$.
Given a free face $\sigma$ of size at most $d$ in $X$, an \emph{elementary $d$-collapse} is the operation of removing from $X$ all the simplices containing $\sigma$ (and contained in $\tau$). We say that $X$ is \emph{$d$-collapsible} if
there is a sequence of elementary $d$-collapses reducing $X$ to the void complex $\emptyset$.
A simplicial complex $X$ is called \emph{$d$-Leray} if all induced subcomplexes of $X$, including $X$ itself, have trivial homology groups in dimensions $d$ and higher.

Wegner showed in \cite{Weg75} that every $d$-representable complex is $d$-collapsible, and that every $d$-collapsible complex is $d$-Leray.
Many interesting results about the intersection patterns of a family of convex sets can be obtained from the fact that the nerve of such a family is $d$-collapsible (or $d$-Leray). For example, in \cite{kalai2005topological}, Kalai and Meshulam proved the following generalizations of Theorem \ref{thm:cht}.

\begin{theorem}[{\cite[Theorem 2.1]{kalai2005topological}}]
\label{thm:kalai_meshulam_collapsible}
    Let $X$ be a $d$-collapsible simplicial complex on vertex set $V$, and let $M$ be a matroid on $V$ of rank $r$, with rank function $\rho$, such that $M\subset X$. Then, there is some $\tau\in X$ such that $\rho(\tau)=r$ and $\rho(V\setminus\tau)\leq d$.
\end{theorem}

\begin{theorem}[{\cite[Theorem 1.6]{kalai2005topological}}]
\label{thm:kalai_meshulam_leray}
     Let $X$ be a $d$-Leray simplicial complex on vertex set $V$, and let $M$ be a matroid on $V$ of rank $r$, with rank function $\rho$, such that $M\subset X$. Then, there is some $\tau\in X$ such that $\rho(V\setminus\tau)\leq d$.
\end{theorem}

Theorem \ref{thm:cht} can be deduced from either Theorem \ref{thm:kalai_meshulam_collapsible} or \ref{thm:kalai_meshulam_leray} by taking $X$ to be the nerve of the family of convex sets, and $M$ to be a ``partition matroid" corresponding to the coloring of the family (see Section \ref{sec:apps} for more details).

Generalizations of Theorem \ref{thm:cht} in a different direction were presented by Arocha et al. in \cite{arocha2009very} and by Montejano and Karasev in \cite{montejano2011topological}.

\begin{theorem}[Very Colorful Helly Theorem {\cite[Theorem 10]{arocha2009very}}]
\label{thm:arocha_et_al}
Let $1\leq k\leq d+1$, and let $\mathcal{C}$ be a finite family of convex sets in $\Rea^d$ colored with $d+1$ colors. If every subfamily $\mathcal{C}'\subset\mathcal{C}$ of size $d+1$ having at least $k$ different colors is intersecting, then there are $d+2-k$ color classes whose union is intersecting.
\end{theorem}

\begin{theorem}[``Semi-intersecting" Colorful Helly Theorem {\cite[Lemma 2]{montejano2011topological}}]
\label{thm:montejano_karasev}
Let $\mathcal{C}$ be finite family of convex sets in $\Rea^d$ colored with $d+2$ colors. If every colorful subfamily of size $d+2$ has at most one non-intersecting subfamily of size $d+1$, then one of the color classes is intersecting.
\end{theorem}

Our main results consist of the following common generalizations of Theorem \ref{thm:arocha_et_al}, Theorem \ref{thm:montejano_karasev} and Theorem \ref{thm:kalai_meshulam_collapsible} (Theorem \ref{thm:kalai_meshulam_leray} respectively).

\begin{theorem}\label{thm:very_colorful_helly_d_coll}

Let $d\geq 1$, $r\geq d+1$, $1\leq m\leq r$ and $m\leq k\leq \min\{m+d,r\}$ be integers. Let $V$ be a finite set with $|V|\geq \max\{m+d,r\}$.
Let $X$ be a $d$-collapsible simplicial complex on vertex set $V$, and let $M$ be a matroid of rank $r$ on vertex set $V$ with rank function $\rho$.

Assume that for every $U=\{u_1,\ldots,u_d,v_1,\ldots,v_m\}\subset V$ with $\rho(U)\geq k$, there is some $i\in[m]$ such that $\{u_1,\ldots,u_d,v_i\}\in X$. 
Then, there is some $\tau\in X$ such that $\rho(\tau)\geq r+1-m$ and $\rho(V\setminus \tau)\leq k-1$.
\end{theorem}

Under the weaker assumption of $d$-Lerayness, we obtain a weaker conclusion:

\begin{theorem}\label{thm:very_colorful_helly_d_leray}
Let $d\geq 1$, $r\geq d+1$, $1\leq m\leq r$ and $m\leq k\leq \min\{m+d,r\}$ be integers. Let $V$ be a finite set with $|V|\geq \max\{m+d,r\}$.
    Let $X$ be a $d$-Leray simplicial complex on vertex set $V$, and let $M$ be a matroid of rank $r$ on vertex set $V$ with rank function $\rho$.

    Assume that for every $U=\{u_1,\ldots,u_{d+1},v_1,\ldots,v_{m-1}\}\subset V$ with $\rho(U)\geq k$, either $\{u_1,\ldots,u_{d+1}\}\in X$ or there exists $j\in[m-1]$ such that $\sigma\cup\{v_j\}\in X$ for every $\sigma\subsetneq \{u_1,\ldots,u_{d+1}\}$.
    Then, there is some $\tau \in X$ such that $\rho(V\setminus\tau) \leq k-1$.
\end{theorem}

\begin{remark}
    Note that the condition in Theorem \ref{thm:very_colorful_helly_d_coll} is implied by the condition in Theorem \ref{thm:very_colorful_helly_d_leray}, and that both conditions are implied by the following stronger (but perhaps more natural) condition: ``For every subset $U$ of $V$ of size $d+m$ with $\rho(U) \geq k$, all but possibly at most $m-1$ of the the subsets of $U$ of size $d+1$ are simplices in $X$".
    Moreover, note that all three conditions are equivalent for $m\leq 2$.
\end{remark}

As an immediate consequence of Theorem \ref{thm:very_colorful_helly_d_coll}, we obtain the following:

\begin{theorem}\label{thm:main_geometric}
    Let $d\geq 1$, $r\geq d+1$, $1\leq m\leq r$ and $m\leq k\leq \min\{m+d,r\}$.
    
    Let $\mathcal{C}$ be a finite family of convex sets in $\Rea^d$, colored with $r$ different colors, of size $|\mathcal{C}|\geq \max\{m+d,r\}$. Assume that for every family $\{A_1,\ldots,A_d,B_1,\ldots,B_m\}\subset \mathcal{C}$,  colored with at least $k$ different colors, at least one of the subfamilies of the form $\{A_1,\ldots,A_d,B_i\}$, for $i\in[m]$, is intersecting.
    Then, there are $r-k+1$ color classes whose union is intersecting.
\end{theorem}

In the special case $m=1$ and $r=d+1$, we recover Theorem \ref{thm:arocha_et_al}. For $m=2$ and $r=k=d+2$, we recover Theorem \ref{thm:montejano_karasev}.

As an application of Theorem \ref{thm:main_geometric}, we obtain the following extension of Tverberg's Theorem \cite{tverberg1966generalization}.

Let $A\subset\Rea^d$ be a finite set of points, and let $A=A_1\cup\ldots\cup A_r$ be a partition of $A$.
We say that $A_1,\ldots, A_r$ is a \emph{Tverberg partition} of $A$ if $\cap_{i=1}^r \conv(A_i)\neq \emptyset$.
Let $B\subset A$. We say that $B\subsetneq A$ is a \emph{Tverberg center} for the partition $A=A_1\cup\cdots\cup A_r$ if, for every $p\in A\setminus B$, 
\[
    \bigcap_{i=1}^r \conv\left((B\cup\{p\})\cap A_i\right)\neq \emptyset.
\]
In other words, $B$ is a Tverberg center for $A=A_1\cup \ldots\cup A_r$ if, for every $p\in A\setminus B$, the partition of $B\cup \{p\}$ induced by $A_1,\ldots,A_r$ is a Tverberg partition.

\begin{theorem}\label{thm:tverberg_sunflower}
Let $d\geq 1$, $r\geq 2$, and let $n=(r-1)(d+1)$.
Let $A\subset \Rea^d$ be a finite set of points of size larger than $n$. Then, there exists a partition $A_1,\ldots,A_r$ of $A$ that has a Tverberg center of size $n$.
\end{theorem}

For $|A|=n+1$, this is exactly Tverberg's theorem. Theorem \ref{thm:tverberg_sunflower} is sharp in the sense that we cannot guarantee a Tverberg center of size smaller than $(r-1)(d+1)$: Let $A$ be a generic set of points in $\Rea^d$ (in the sense that the $d|A|$ coordinates of $A$ are algebraically independent over the rationals). Assume for contradiction that $A$ has a partition $A_1,\ldots, A_r$ with a Tverberg center $B$ of size $k<(r-1)(d+1)$. Then, for every $p\in A\setminus B$, $B'=B\cup\{p\}$ is a generic set of $k+1\leq (r-1)(d+1)$ points in $\Rea^d$, with a Tverberg partition into $r$ parts. However, it is well know that such a set does not have a Tverberg partition into $r$ parts (see e.g. \cite{tverberg1966generalization}). In fact, the condition of genericity can be replaced by the weaker property of ``strong general position" introduced by Reay in \cite{reay1968extension} (under the name ``strong independence"); see also  \cite{perles2014strong,perles2017tverberg}.

Theorem \ref{thm:tverberg_sunflower} is closely related to the ``tolerant Tverberg Theorem" (see e.g. \cite{larman1972sets,colin2007applying,soberon2012generalisation,garcia2015projective,garcia2017note}). Using Theorem \ref{thm:tverberg_sunflower}, we obtain a new proof of the following result from \cite{soberon2012generalisation}:

\begin{theorem}[Sober{\'o}n-Strausz {\cite[Theorem 1]{soberon2012generalisation}}]
\label{thm:soberon_strausz}
  Let $d\geq 1$, $r\geq 2$ and $t\geq 0$. Let $A\subset\Rea^d$ be a finite set of points of size at least $(t+1)(r-1)(d+1)+1$. Then, there exists a partition $A_1,\ldots,A_r$ of $A$ such that, for any $C\subset A$ of size $t$,
  \[
    \bigcap_{i=1}^r \conv(A_i\setminus C)\neq\emptyset.
  \]
\end{theorem}

Our paper is organized as follows. In Section \ref{sec:prelims} we present some basic facts about simplicial complexes and matroids that we will use later. In Section \ref{sec:tolerance_matroids} we introduce the notion of ``tolerance complexes" of matroids, which is used in the proofs of our main results. In Section \ref{sec:dcoll} we present the proof of Theorem \ref{thm:very_colorful_helly_d_coll}, and in Section \ref{sec:dleray} we present the proof of Theorem \ref{thm:very_colorful_helly_d_leray}.
Section \ref{sec:apps} deals with the geometric applications of our results, and contains the proofs of Theorems \ref{thm:main_geometric}, \ref{thm:tverberg_sunflower} and \ref{thm:soberon_strausz}. We conclude in Section \ref{sec:conclusion} with some open problems arising from our work.

\section{Preliminaries}\label{sec:prelims}

\subsection{Simplicial complexes}

A \emph{simplicial complex} $X$ is a family of subsets of some finite set $V$, such that for every $\sigma\in X$ and $\sigma'\subset \sigma$, we have $\sigma'\in X$. The set $V = V(X)$ is called the \emph{vertex set} of $X$. The sets in $X$ are called the \emph{simplices} or \emph{faces} of the complex. The \emph{dimension} of a simplex $\sigma\in X$ is  $|\sigma|-1$, and the dimension of $X$, denoted by $\dim(X)$, is the maximal dimension of a simplex in $X$.
We denote by $2^V=\{\sigma:\, \sigma\subset U\}$ the \emph{complete complex} on vertex set $V$.

A \emph{subcomplex} of $X$ is a simplicial complex $X'$ such that each simplex of $X'$ is also a simplex of $X$.
For $U\subset V$, the subcomplex of $X$ \emph{induced} by $U$ is the complex
\[
X[U]=\{\sigma\in X:\, \sigma\subset U\}.
\]
Let $\sigma$ be a simplex in $X$.
We define the \emph{link} of $\sigma$ in $X$ to be the subcomplex
\[
\lk(X,\sigma)= \{\tau\in X: \, \sigma\cap\tau=\varnothing,\, \sigma\cup\tau \in X\},
\] 
and the \emph{costar} of $\sigma$ in $X$ to be the subcomplex
\[
\cost(X,\sigma) = \{\tau\in X:\, \sigma \not\subset\tau\}.
\]
If $\sigma$ is a $0$-dimensional simplex, i.e. if $\sigma = \{v\}$ for some vertex $v \in V$, we simply write $\lk(X,v)=\lk(X,\{v\})$, 
and $X\setminus v= \cost(X,\{v\})=X[V\setminus \{v\}]$.

Let $X$ and $Y$ be simplicial complexes on disjoint vertex sets.
The \emph{join} of $X$ and $Y$ is the simplicial complex on $V(X) \cup V(Y)$ defined as
\[
X \ast Y= \{ \sigma\cup \tau:\,  \sigma\in X, \tau\in Y\}.
\]

The following well known special case of the Mayer-Vietoris exact sequence relates the homology of a complex $X$ to that of the link and costar of one of its vertices (see e.g. \cite[Theorem 2.2]{kim2021leray} for a detailed proof).

\begin{theorem}\label{thm:exact_sequence_link_costar}
Let $X$ be a simplicial complex on vertex set $V$, and let $v\in V$. Then, there is an exact sequence
\[
\cdots\to
\zhomology{k}{\lk(X,v)} \to \zhomology{k}{X\setminus v}\to \zhomology{k}{X}\to \zhomology{k-1}{\lk(X,v)}\to \cdots
\]
\end{theorem}

We will need the following equivalent definition of $d$-Leray complexes, due to Kalai and Meshulam.

\begin{lemma}[{\cite[Proposition 3.1]{kalai2006intersections}}]\label{lemma:km_dleray}
A simplicial complex $X$ is $d$-Leray if and only if $\tilde{H}_k(\lk(X,\sigma))=0$ for all $\sigma\in X$ and $k\geq d$.  
\end{lemma}

Recall that, given a free face $\sigma\in X$ of size $|\sigma|\leq d$, an elementary $d$-collapse is the operation consisting of removing $\sigma$ and all simplices containing it from $X$. Note that the complex obtained after this operation is exactly $\cost(X,\sigma)$. We denote such an elementary $d$-collapse by $X\to X'=\cost(X,\sigma)$.
The following equivalent definition of $d$-collapsibility was observed by Tancer in \cite{tancer2010d} (see also \cite[Lemma 1]{Weg75} for a similar statement):

\begin{lemma}[{\cite[Lemma 5.2]{tancer2010d}}]\label{lemma:tancer}
    A simplicial complex $X$ is $d$-collapsible if and only if there is a sequence of elementary $d$-collapses \[X=X_1\to X_2\to \cdots\to X_t,\] such that the free face in each elementary $d$-collapse is of size exactly $d$, and $\dim(X_t)< d-1$.
\end{lemma}

\begin{comment}
\begin{theorem}[Wegner \cite{Weg75}]\label{thm:wegner}
    Let $C_1,\ldots, C_m$ be finite families of convex sets in $\Rea^d$. Then, $N(C_1,\ldots,C_m)$ is $d$-collapsible.
\end{theorem}

\begin{corollary}\label{cor:barany_wegner}
Let $A\subset\Rea^d$ be a finite set of points. Then, the complex
\[
    X= \{\sigma\subset A:\, 0\notin\conv(\sigma)\}
\]
is $d$-collapsible.
\end{corollary}
\todo[inline]{Add citation to corollary and its proof (essentially Barany's paper?), or include short proof}
\end{comment}

\subsection{Matroids}

A family $M$ of subsets of a non-empty set $V$ is a {\em matroid} if it satisfies
\begin{itemize}
    \item[(i)] $\emptyset \in M$, 
    \item[(ii)] for all $A' \subset A \subset V$, if $A \in M$ then $A' \in M$, and
    \item[(iii)] if $A, B \in M$ and $|A| < |B|$, then there exists $x \in B \setminus A$ such that $A \cup \{x\} \in M$.
\end{itemize}

The elements of $M$ are usually called the \emph{independent sets} of $M$, and the maximal elements are called the \emph{bases} of $M$.
The {\em rank function} of a matroid $M$ on $V$ is a function $\rho: 2^V \to \mathbb{N}$ such that for every $W \subset V$, $\rho(W)$ equals the maximal size of a set $W' \subset W$ with $W' \in M$. The \emph{rank of the matroid $M$} is defined as $\rho(V)$.
The \emph{span} of a set $U\subset V$
 is defined as
 \[
 \text{span}_M(U)= \{v\in V:\, \rho(U\cup\{v\})=\rho(U)\}.
 \]

Note that the conditions (i) and (ii) allow us to regard a matroid $M$ as a simplicial complex. Moreover, condition (iii), sometimes called the ``exchange property" for independent sets, implies that $M$ is a pure simplicial complex - that is, all the bases of $M$ are of the same size (equal to the rank of $M$). See e.g. \cite{oxley2006matroid} for more background on matroids.

We will need the following very simple auxiliary results about matroids.

\begin{lemma}\label{lemma:set_U}
Let $M$ be a rank $r$ matroid on vertex set $V$, and let $\rho$ be its rank function. Then, for each $0\leq t\leq |V|$, there exists $U\subset V$ such that $|U|=t$ and $\rho(U)\geq \min \{t,r\}$.
\end{lemma}
\begin{proof}
If $t\leq r$ let $U$ be any independent set of $M$ of size $t$. If $t\geq r$ let $U\subset V$ be any set of size $t$ containing a basis of $M$.
\end{proof}

\begin{lemma}\label{lemma:matroid_rank_subset}
Let $M$ be a matroid  on vertex set $V$ with rank function $\rho$. Let $U\subset W\subset V$. Then
\[
    \rho(U)\geq \rho(W)-|W\setminus U|.
\]
\end{lemma} 
\begin{proof}
By definition of the rank function, $W$ contains an independent set $W'$ of size $\rho(W)$. Then, $W'\cap U$ is an independent set of size at least $|W'|-|W\setminus U|=\rho(W)-|W\setminus U|$. So $\rho(U)\geq \rho(W)-|W\setminus U|$.
\end{proof}

\section{Tolerance complexes of matroids}\label{sec:tolerance_matroids}

Let $M$ be a matroid with rank function $\rho$ on vertex set $V$. Let $r=\rho(V)$ be the rank of $M$. For $0\leq t\leq r$, let 
\[
    M^t= \{\sigma\subset V:\, \rho(\sigma)\geq |\sigma|-t\}.
\]
In particular, for $t=0$ we have $M^0=M$. Note that $M^t$ is exactly the ``$t$-tolerance complex" of $M$, as defined in \cite{kim2021leray}.

\begin{lemma}\label{lem:induced_matroid}
Let $0\leq t\leq r$. Then, $M^t$ is a matroid.
\end{lemma}
\begin{proof}
First, note that $M^t$ is a simplicial complex: if $\sigma\in M^t$ and $\sigma'\subset \sigma$, then, by Lemma \ref{lemma:matroid_rank_subset},
\[
    \rho(\sigma')\geq \rho(\sigma)-|\sigma\setminus\sigma'|
    \geq |\sigma|-|\sigma\setminus\sigma'|-t =|\sigma'|-t,
\]
and therefore $\sigma'\in M^t$. 

We will show that $M^t$ satisfies the exchange property for independent sets: Let $\sigma,\tau\in M^t$ such that $|\tau|>|\sigma|$. We have to show that there is a vertex $v\in \tau\setminus \sigma$ such that $\sigma\cup\{v\}\in M^t$. 

Since $\sigma\in M^t$, we have $\rho(\sigma)\geq |\sigma|-t$.
If $\rho(\sigma)>|\sigma|-t$, then for all $v\in V\setminus \sigma$ we have $\rho(\sigma\cup\{v\})\geq \rho(\sigma) \geq |\sigma\cup \{v\}|-t$, as wanted. If $\rho(\sigma)=|\sigma|-t$, then, since $\tau\in M^t$,
\[
    \rho(\tau)\geq |\tau|-t > |\sigma|-t =\rho(\sigma).
\]
Let $\tau'\subset \tau$ such that $|\tau'|=\rho(\tau)$ and $\tau'\in M$, and let $\sigma'\subset \sigma$ such that $|\sigma'|=\rho(\sigma)$ and $\sigma'\in M$. Then, since $|\tau'|>|\sigma'|$, there exists $v\in \tau'\setminus \sigma'$ such that $\sigma'\cup\{v\}\in M$.
Note that $v\in \tau\setminus \sigma$ (otherwise, $\sigma'\cup \{v\}\subset \sigma$, in contradiction to $\rho(\sigma)=|\sigma'|$).
We obtain
\[
\rho(\sigma\cup\{v\})\geq \rho(\sigma)+1 =|\sigma\cup\{v\}|-t,
\]
and therefore, $\sigma\cup\{v\}\in M^t$. Hence, $M^t$ is a matroid.
\end{proof}

\begin{lemma}
\label{lemma:induced_matroid_rank}
Let $0\leq t\leq r$. Denote the rank function of $M^t$ by $\rho^t$. Then, for every $A\subset V$,
\[
    \rho^t(A)= \min\{ |A|,\rho(A)+t\}.
\]
\end{lemma}
\begin{proof}

Assume first that $|A|\leq \rho(A)+t$.  Then $\rho(A) \geq |A|-t$, and therefore $A\in M^t$. In particular, $\rho^t(A)=|A|$.

Assume now that $|A|\geq \rho(A)+t$. Let $I\subset A$ such that $|I|=\rho(A)$ and $I\in M$. Then, any $U\subset A$ of size $\rho(A)+t$ containing $I$ satisfies
\[
\rho(U)\geq \rho(I)=\rho(A) = (\rho(A)+t)-t= |U|-t,
\]
and therefore belongs to $M^t$. Thus, $\rho^t(A)\geq \rho(A)+t$. On the other hand, every $U\subset A$ such that $U\in M^t$ satisfies $\rho(U)\geq |U|-t$, and therefore $|U|\leq \rho(U)+t\leq \rho(A)+t$. So $\rho^t(A)\leq \rho(A)+t$.
\end{proof}

\begin{remark}
    Given two matroids $M_1$ and $M_2$ on vertex set $V$, the \emph{matroid union} of $M_1$ and $M_2$ is defined as $M_1\vee M_2=\{\sigma\cup\tau:\, \sigma\in M_1,\, \tau\in M_2\}$. The matroid union theorem (see e.g. \cite[Sec. 11.3]{oxley2006matroid}) states that $M_1\vee M_2$ is a matroid, with rank function $\rho_{12}$ satisfying
    \[
        \rho_{12}(A)=\min\{\rho_1(B)+\rho_2(B)+|A\setminus B| :\, B\subset A\}
    \]
    for every $A\subset V$, where $\rho_1$ and $\rho_2$ are the rank functions of $M_1$ and $M_2$ respectively.
    
    Note that $M^t$ is the matroid union of $M$ and the uniform matroid $\{A\subset V:\, |A|\leq t\}$. 
    Therefore, Lemma \ref{lem:induced_matroid} and Lemma \ref{lemma:induced_matroid_rank} can also be obtained as a consequence of the matroid union theorem.
\end{remark}

\section{Proof of Theorem \ref{thm:very_colorful_helly_d_coll}}\label{sec:dcoll}

In this section we prove Theorem \ref{thm:very_colorful_helly_d_coll}. Our proof extends the argument for the case $m=1, k=d+1$ given by Kalai and Meshulam in \cite[Theorem 2.1] {kalai2005topological}.

\begingroup
\def\thetheorem{\ref{thm:very_colorful_helly_d_coll}}
\begin{theorem}
Let $d\geq 1$, $r\geq d+1$, $1\leq m\leq r$ and $m\leq k\leq \min\{m+d,r\}$ be integers. Let $V$ be a finite set with $|V|\geq \max\{m+d,r\}$.
Let $X$ be a $d$-collapsible simplicial complex on vertex set $V$, and let $M$ be a matroid of rank $r$ on vertex set $V$ with rank function $\rho$.

Assume that for every $U=\{u_1,\ldots,u_d,v_1,\ldots,v_m\}\subset V$ with $\rho(U)\geq k$, there is some $i\in[m]$ such that $\{u_1,\ldots,u_d,v_i\}\in X$. 
Then, there is some $\tau\in X$ such that $\rho(\tau)\geq r+1-m$ and $\rho(V\setminus \tau)\leq k-1$.
\end{theorem}
\addtocounter{theorem}{-1}
\endgroup
\begin{proof}[Proof of Theorem \ref{thm:very_colorful_helly_d_coll}]
Let
\begin{align*}
    \mathcal{R} &=\{\eta\in X:\, |\eta|\geq d+1,\, \rho(\eta)\geq k-m-d+|\eta|\} \\&= \{\eta\in X\cap M^{d+m-k} : \, |\eta|\geq d+1\}.
\end{align*}

By Lemma \ref{lemma:set_U}, there is some  $U=\{u_1,\ldots,u_d,v_1,\ldots,v_{m}\}\subset V$ with  $\rho(U)\geq \min\{m+d,r\}\geq k$. Then, there is some $i\in[m]$ such that $\sigma_i=\{u_1,\ldots,u_d,v_i\}\in X$. Note that $\sigma_i\subset U\in M^{d+m-k}$, so $\sigma_i\in \mathcal{R}$, and in particular $\mathcal{R}\neq \emptyset$.

By Lemma \ref{lemma:tancer}, there is a sequence of elementary $d$-collapses $X=X_1\to X_2\to\cdots\to X_t$, such that the free faces collapsed at each step are of size exactly $d$, and $X_t$ satisfies $\dim(X_t)< d-1$.

By definition, $\mathcal{R}\subset X=X_1$, and clearly $\mathcal{R}\not\subset X_t$. Let $1\leq n<t$ be the maximal index such that $\mathcal{R}\subset X_n$.
Let $\sigma$ be the free face of size $d$ in the elementary $d$-collapse $X_n\to X_{n+1}$, and let $\tau$ be the unique maximal face of $X_n$ containing $\sigma$.

Since $\mathcal{R}\subset X_n$ but $\mathcal{R}\not\subset X_{n+1}$,  there is some $\sigma\subset\sigma'\subset\tau$ such that $\sigma'\in M^{d+m-k}$. In particular, we have $\sigma\in M^{d+m-k}$. 
Let
\[
    \mathcal{A}= \{ \eta\subset V\setminus \sigma :\, \sigma\cup \eta \in M^{d+m-k}, \, \sigma\cup\{u\}\notin X_n \text{ for all } u\in \eta \}.
    \]
Note that, since $\sigma\in M^{d+m-k}$, we have $\emptyset\in\mathcal{A}$, and therefore $\mathcal{A}\neq\emptyset$. Let $\eta$ be a maximal element of $\mathcal{A}$. 

First, note that $|\eta|\leq m-1$. Otherwise, let  $v_1,\ldots,v_m\in \eta$, and let $U=\sigma\cup\{v_1,\ldots,v_m\}$.
Then $U \subset \sigma\cup \eta\in M^{d+m-k}$, and therefore
\[
    \rho(U)\geq |U|-(d+m-k)= (d+m)-(d+m-k)=k.
\]
Therefore, there is some $i\in[m]$ such that $\sigma_i=\sigma\cup\{v_i\}\in X$. Since $\sigma_i\subset \sigma\cup\eta\in M^{d+m-k}$, we also have $\sigma_i\in M^{d+m-k}$.
Thus, $\sigma_i\in \mathcal{R}$, and therefore 
$\sigma_i=\sigma\cup\{v_i\}\in X_n$, 
in contradiction to $\eta\in \mathcal{A}$.

Next, note that, for all $v\in V\setminus \text{span}_M(\sigma\cup\eta)$, we have $\sigma\cup\{v\}\in X_n$. Otherwise, we would have
\[
    \rho(\sigma\cup\eta\cup\{v\})=\rho(\sigma\cup\eta)+1
    \geq |\sigma\cup\eta|-(d+m-k)+1    
    = |\sigma\cup\eta\cup\{v\}|-(d+m-k).
\]
So $\sigma\cup\eta\cup\{v\}\in M^{d+m-k}$,
and therefore $\eta\cup\{v\}\in\mathcal{A}$, in contradiction the the maximality of $\eta$. Therefore,
since $\sigma$ is contained in the unique maximal face $\tau$ of $X_n$, we obtain
   \begin{equation}\label{eq:tau1}
        V\setminus \text{span}_M(\sigma\cup\eta) \subset \tau,
    \end{equation}
    or equivalently,
    \[
        V\setminus \tau \subset \text{span}_M(\sigma\cup\eta).
    \]
  We divide into two cases: If $\rho(\sigma\cup\eta)=|\sigma\cup\eta|-(d+m-k)= k-m+|\eta|$, then
  \[
    \rho(V\setminus \tau)\leq \rho(\sigma\cup\eta) = k-m+|\eta|\leq k-1,
  \]
  as wanted.
  
  If $\rho(\sigma\cup\eta)\geq  |\sigma\cup\eta|-(d+m-k)+1$, then for every $v\in V\setminus(\sigma\cup\eta)$, we have
  \[
  \rho(\sigma\cup\eta\cup\{v\})\geq \rho(\sigma\cup\eta) \geq  |\sigma\cup\eta\cup\{v\}| -(d+m-k).
  \]
  So $\sigma\cup\eta\cup\{v\}\in M^{d+m-k}$, and therefore, by the maximality of $\eta$ in $\mathcal{A}$, we must have $\sigma\cup\{v\}\in X_n$. 
  Hence, since $\sigma$ is contained in the unique maximal face $\tau$ of $X_n$, we have $V\setminus\eta \subset \tau$ (actually, $V\setminus\eta= \tau$, since $\sigma\cup\{u\}\notin X_n$ for all $u\in \eta$ by the definition of $\mathcal{A}$). Thus,
  \[
    \rho(V\setminus \tau)= \rho(\eta) \leq |\eta|\leq m-1 \leq k-1.
  \]

  Finally, let $v\in V$. If $v\in \text{span}_M(\sigma\cup \eta)$ then, since $\sigma\subset \tau$, we have $v\in \text{span}_M(\tau\cup \eta)$. If $v\notin \text{span}_M(\sigma\cup \eta)$, then by \eqref{eq:tau1}, $v\in\tau\subset \text{span}_M(\tau\cup\eta)$. Thus, $\text{spam}_M(\tau\cup\eta)=V$, and therefore $\rho(\tau\cup\eta)=\rho(V)=r$. Hence, by Lemma \ref{lemma:matroid_rank_subset}, we obtain
  \[
    \rho(\tau)\geq r-|\eta|\geq r-m+1.
  \]
\end{proof}

\section{Proof of Theorem \ref{thm:very_colorful_helly_d_leray}}\label{sec:dleray}

First, we prove the $m=1$ case of Theorem \ref{thm:very_colorful_helly_d_leray}:

\begin{theorem}\label{thm:very_colorful_helly_d_leray_m==1}
Let $X$ be a $d$-Leray simplicial complex on vertex set $V$. Let $M$ be a matroid on $V$ of rank $r\geq d+1$. Denote by $\rho$ the rank function of $M$. Let $1\leq k\leq d+1$.

Assume that $X$ contains all the sets $\sigma\subset V$ of size $d+1$ such that $\rho(\sigma)\geq k$.
Then, there is some $\tau\in X$ such that $\rho(V\setminus \tau)\leq k-1$.
\end{theorem}

The main tool we will need for the proof is the following result by Kalai and Meshulam (implicit in the proof of \cite[Thm 1.6]{kalai2005topological}; see also \cite[Theorem 3]{aharoni2022cooperative} for a similar statement and \cite[Theorem 4.5]{aharoni2006intersection} for an equivalent dual version).

\begin{lemma}[Kalai-Meshulam {\cite{kalai2005topological}}]
\label{lemma:kalai_meshulam}
Let $X$ be a simplicial complex and $M$ a matroid with rank function $\rho$ on the same vertex set $V$. For each $\sigma\in X$, let \[
\ell_{\sigma}=\min\{k\geq -1:\, \tilde{H}_i(\lk(X,\sigma))=0\,\, \forall i\geq k\}.
\]
If for all $\sigma\in X$,
\[
   \rho(V\setminus \sigma)\geq \ell_{\sigma}+1,
\]
then $M\not\subset X$.
\end{lemma}

We will also need the following well known simple result. 

\begin{lemma}[Helly's Theorem for $d$-Leray complexes]
\label{lemma:helly_d_leray}
Let $X$ be a $d$-Leray complex on vertex set $V$, and let $U\subset V$ be a set of size at least $d+1$. Assume that every subset of $U$ of size $d+1$ is a simplex of $X$. Then, $U\in X$.
\end{lemma}
\begin{proof}
Assume for contradiction that $U\notin X$. Let $S\subset U$ be an inclusion minimal set such that $S\notin X$. By the assumption of the lemma, we must have $|S|\geq d+2$. But then, by the minimality of $S$, $X[S]$ is just the boundary of a $(|S|-1)$-dimensional simplex, which has non-trivial homology in dimension $|S|-2 \geq d$. This is a contradiction to the assumption that $X$ is $d$-Leray.
\end{proof}

\begin{proof}[Proof of Theorem \ref{thm:very_colorful_helly_d_leray_m==1}]
Assume for contradiction that for all $\sigma\in X$, $\rho(V\setminus\sigma)\geq k$.
In particular, we have $|V\setminus \sigma|\geq k$ for all $\sigma\in X$.

Recall that we defined, for $\sigma\in X$,
\[
\ell_{\sigma}=\min\{n\geq -1:\, \tilde{H}_i(\lk(X,\sigma))=0\, \forall i\geq n\}.
\]

Let $t=d+1-k$. In order to apply Lemma \ref{lemma:kalai_meshulam}, we will show that $\rho^t(V\setminus\sigma)\geq \ell_{\sigma}+1$ for all $\sigma\in X$.

Let $\sigma\in X$. By Lemma \ref{lemma:induced_matroid_rank}, we have 
\[
\rho^t(V\setminus\sigma)\geq
\min\{|V\setminus\sigma|,t+k\}=
\min\{|V\setminus\sigma|,d+1\}.
\]
We divide into two cases:

First, assume $|V\setminus\sigma|\geq d+1$. Then, since $X$ is $d$-Leray, we have, by Lemma \ref{lemma:km_dleray}, $\ell_{\sigma}\leq d$,  and therefore
\[
  \rho^t(V\setminus\sigma)\geq \min\{|V\setminus\sigma|,d+1\}= d+1 \geq \ell_{\sigma}+1.
\]

Now, assume $|V\setminus\sigma|\leq d$. Then, we have $\rho^t(V\setminus\sigma)=|V\setminus\sigma|$.
Since $|V\setminus\sigma'|\geq k$ for all $\sigma'\in X$, we must have
$\dim(X)\leq |V|-k-1$, and therefore
\[
\dim(\lk(X,\sigma))\leq |V\setminus\sigma|-k-1.
\]
In particular, $\ell_{\sigma}\leq |V\setminus\sigma|-k$.
We obtain
\[
    \rho^t(V\setminus\sigma)=|V\setminus\sigma|\geq \ell_{\sigma}+k\geq \ell_{\sigma}+1.
\]
Hence, by Lemma \ref{lemma:kalai_meshulam}, we have $M^t\not\subset X$. That is, there is some $\tau\in M^t$ such that $\tau\notin X$. If $|\tau|\geq d+1$  then, by Lemma \ref{lemma:helly_d_leray}, there is some $\tau'\subset \tau$ of size $d+1$ such that $\tau'\notin X$. 
If $|\tau|<d+1$ then, since $M^t$ is a matroid of rank $r\geq d+1>|\tau|$, there is some $\tau'\supset \tau$ of size $d+1$ such that $\tau'\in M^t$, and, since $\tau\notin X$, we also have $\tau'\notin X$.

In both cases, since $\tau'\in M^t$, we have
\[
    \rho(\tau')\geq |\tau'|-t= d+1-t= k,
\]
but this is a contradiction to the assumption of the theorem.
\end{proof}

For the proof of Theorem \ref{thm:very_colorful_helly_d_leray}, will need the following Lemma about $d$-Leray complexes.

\begin{lemma}\label{lem:U_is_a_simplex}
    Let $X$ be a $d$-Leray complex on vertex set $V$, and let $A\subset d+1$ such that $A\notin X$. Let
    \[
    U=\{ v\in V\setminus A:\, v\cup \sigma\in X:\, \text{ for all } \sigma\subset A, |\sigma|=d\}.
    \]
    Assume that $U\neq \emptyset$.
    Then,
    for every $a\in A$, $U\cup A\setminus \{a\}\in X$.
\end{lemma}
\begin{proof}
Given a set $B\subset V$ and $i\geq 0$, we denote by $\binom{B}{i}$ the family of all subsets of $B$ of size $i$, and we denote by $\partial 2^B$ the boundary of the simplex on vertex set $B$, i.e. the simplicial complex consisting of all proper subsets of $B$.

First note that, since $U\neq\emptyset$, we must have $\sigma\in X$ for every $\sigma\subsetneq A$. That is, $\partial 2^A\subset X$.
We will show that, for every $1\leq i\leq |U|$ and $W \in \binom{U}{i}$, $X[A\cup W] =  \partial 2^A\ast 2^{W}$.
We argue by induction on $i$. The base case $i=1$ is obvious by the definition of $U$.
Suppose, for some $1\leq i<|U|$, that $X[A\cup T] =\partial 2^A\ast  2^{T}$ for every $T \in \binom{U}{i}$.
Let $W = \{u_1,u_2,\ldots,u_{i+1}\} \subset U$.

Consider \[L = \lk(X[A\cup W], \{u_1,u_2,\ldots,u_{i-1}\}).\]
Note that, by Lemma \ref{lemma:km_dleray}, $\tilde{H}_j(L) = 0$ for every $j \geq d$. Furthermore,
by the induction hypothesis, we have $X[A\cup (W\setminus \{u_{i}\})]=\partial 2^A \ast 2^{W\setminus \{u_{i}\}}$, and therefore $L \setminus u_{i} = \partial 2^A\ast 2^{\{u_{i+1}\}}$. In particular, $L\setminus u_i$ contractible, and thus $\tilde{H}_j(L\setminus u_i)=0$ for all $j\geq 0$.
By Theorem \ref{thm:exact_sequence_link_costar}, there exists a long exact sequence
\[
    \cdots \to \tilde{H}_j(L)\to \tilde{H}_{j-1}(\lk(L,u_i))\to \tilde{H}_{j-1}(L\setminus u_i)\to\cdots
\]
We obtain that \[\tilde{H}_j(\lk(L, u_i)) = 0\;\;\text{for all}\;\;j \geq d-1.\]
Since $\partial 2^A \subset \lk(L, u_i)=\lk(X[A\cup W], \{u_1,u_2,\ldots,u_{i}\})$, there must be a $d$-dimensional chain in $\lk(L, u_i)$ whose boundary equals to $\partial 2^A$.
For each $(d-1)$-dimensional simplex $\sigma$ in $\partial 2^A$, the only possible $d$-dimensional simplex of $\lk(L, u_i)$ containing $\sigma$ is $\sigma\cup \{u_{i+1}\}$. This shows $\partial 2^A \ast 2^{\{u_{i+1}\}}\subset  \lk(L,u_i)$. Since $A\notin \lk(L,u_i)$, we must actually have
$\lk(L, u_i) = \partial 2^A\ast 2^{\{u_{i+1}\}} $, and hence we have $X[A\cup W] = \partial 2^A \ast 2^W$.

Finally, letting $i=|U|$, we obtain that $\partial 2^A\ast 2^U \subset X$, as desired.
\end{proof}

% restating Theorem 1.7
\begingroup
\def\thetheorem{\ref{thm:very_colorful_helly_d_leray}}
\begin{theorem}
Let $d\geq 1$, $r\geq d+1$, $1\leq m\leq r$ and $m\leq k\leq \min\{m+d,r\}$ be integers. Let $V$ be a finite set with $|V|\geq \max\{m+d,r\}$.
    Let $X$ be a $d$-Leray simplicial complex on vertex set $V$, and let $M$ be a matroid of rank $r$ on vertex set $V$ with rank function $\rho$.
    
     Assume that for every $U=\{u_1,\ldots,u_{d+1},v_1,\ldots,v_{m-1}\}\subset V$ with $\rho(U)\geq k$, either $\{u_1,\ldots,u_{d+1}\}\in X$ or there exists $j\in[m-1]$ such that $\sigma\cup\{v_j\}\in X$ for every $\sigma\subsetneq \{u_1,\ldots,u_{d+1}\}$.
    Then, there is some $\tau \in X$ such that $\rho(V\setminus\tau) \leq k-1$.
\end{theorem}
\addtocounter{theorem}{-1}
\endgroup
\begin{proof}
We argue by induction on $m$. For $m=1$ the claim holds by Theorem~\ref{thm:very_colorful_helly_d_leray_m==1}. Let $m>1$ and assume that the claim holds for $m-1$. 

If for all $U=\{u_1,\ldots,u_{d+1},v_1,\ldots,v_{m-2}\}\subset V$ with $\rho(U)\geq k-1$, either $\{u_1,\ldots,u_{d+1}\}\in X$ or there is some $1\leq j\leq m-2$ such that $\sigma\cup\{v_j\}\in X$ for all $\sigma\subsetneq \{u_1,\ldots,u_{d+1}\}$, then by the induction hypothesis there is some $\tau\in X$ such that $\rho(V\setminus \tau)\leq k-2\leq k-1$, as wanted.

Otherwise, there exists $U=\{u_1,\ldots,u_{d+1},v_1,\ldots,v_{m-2}\}\subset V$ with $\rho(U)\geq k-1$, such that $A=\{u_1,\ldots,u_{d+1}\}\notin X$ and for every $1\leq j\leq m-2$ there is some $\sigma_j\subsetneq A$ such that $\sigma_j\cup\{v_j\}\notin X$.
We divide into two cases: First, assume that $\rho(U)=k-1$. Then, for any $v\in V\setminus \text{span}_M(U)$, we have $\rho(\{u_1,\ldots,u_{d+1},v_1,\ldots,v_{m-2},v\})\geq k$, and therefore by the assumption of the theorem we must have $\sigma\cup\{v\}\in X$ for all $\sigma\subsetneq A$. Note that, since $\rho(V)=r \geq k >\rho(U)$, we must have $V\setminus \text{span}_M(U)\neq \emptyset$.
Hence, by Lemma~\ref{lem:U_is_a_simplex}, we have $\tau=V\setminus \text{span}_M(U)\in X$. Note that $\rho(V\setminus \tau)=\rho(U)= k-1$, as desired.

Next, assume that $\rho(U)\geq k$. Then, for any $v\in V\setminus U$, we have $\rho(\{u_1,\ldots,u_{d+1},v_1,\ldots,v_{m-2},v\})\geq\rho(U)\geq k$, and therefore by the assumption of the theorem, we must have $\sigma\cup\{v\}\in X$ for every $\sigma\subsetneq A$. Note that, since $|V|\geq d+m \geq |U|+1$, we must have $V\setminus U\neq \emptyset$.
So, by Lemma~\ref{lem:U_is_a_simplex}, we have $\tau=(V\setminus U)\cup \{u_1,\ldots,u_d\}\in X$. Since $\rho(V\setminus \tau)=\rho(\{u_{d+1},v_1,\ldots,v_{m-2})\leq m-1\leq k-1$, we are done.
\end{proof}

\section{Geometric applications}\label{sec:apps}

\subsection{Colorful Helly and Carath\'eodory theorems}

In this section we present a proof of Theorem \ref{thm:main_geometric}, and show how to derive from it an analogous extension of B\'ar\'any's  Colorful Carath\'eodory Theorem \cite{barany1982generalization}. The arguments in this section are standard and well known (see e.g. \cite{kalai2005topological,barany1982generalization}), but we include detailed proofs for completeness.

\begingroup
\def\thetheorem{\ref{thm:main_geometric}}
\begin{theorem}
    Let $d\geq 1$, $r\geq d+1$, $1\leq m\leq r$ and $m\leq k\leq \min\{m+d,r\}$.
    
    Let $\mathcal{C}$ be a finite family of convex sets in $\Rea^d$, colored with $r$ different colors, of size $|\mathcal{C}|\geq \max\{m+d,r\}$.
    Assume that for every family $\{A_1,\ldots,A_d,B_1,\ldots,B_m\}\subset \mathcal{C}$,  colored with at least $k$ different colors, at least one of the subfamilies of the form $\{A_1,\ldots,A_d,B_i\}$, for $i\in[m]$, is intersecting.
    Then, there are $r-k+1$ color classes whose union is intersecting.
\end{theorem}
\addtocounter{theorem}{-1}
\endgroup
\begin{proof}
Let $X$ be the nerve of the family $\mathcal{C}$. By Wegner's Theorem \cite{Weg75}, $X$ is $d$-collapsible.
Let $\mathcal{C}_1,\ldots,\mathcal{C}_r$ be the color classes of $\mathcal{C}$. Let $M$ be the partition matroid on vertex set $\mathcal{C}$,  corresponding to the partition $\mathcal{C}=\mathcal{C}_1\cup\cdots\cup\mathcal{C}_r$. That is, the independent sets of $M$ are exactly the colorful subfamilies of $\mathcal{C}$, and the rank function $\rho$ satisfies
\[
    \rho(\mathcal{C}')= |\{i\in[r]:\, \mathcal{C}'\cap \mathcal{C}_i\neq \emptyset\}| 
\]
for all $\mathcal{C}'\subset\mathcal{C}$. That is, $\rho(\mathcal{C}')$ is the number of different colors appearing in $\mathcal{C}'$.

Let $U=\{A_1,\ldots,A_d,B_1,\ldots,B_m\}\subset V$ be a subfamily colored with at least $k$ different colors (or equivalently, satisfying $\rho(U)\geq k$). Then, there is some $i\in[m]$ such that $\{A_1,\ldots,A_d,B_i\}$ is intersecting. In other words, $\{A_1,\ldots,A_d,B_i\}\in X$. Therefore, by Theorem \ref{thm:very_colorful_helly_d_coll}, there is some $\tau\in X$ such that $\rho(V\setminus \tau)\leq k-1$. That is, $\tau\subset\mathcal{C}$ is an intersecting subfamily, whose complement intersects at most $k-1$ color classes. This means that $\tau$ contains at least $r-k+1$ entire color classes, as wanted.
\end{proof}

\begin{theorem}\label{thm:caratheodory}
  Let $d\geq 1$, $r\geq d+1$, $1\leq m\leq r$ and $m\leq k\leq \min\{m+d,r\}$.
  Let $P$ be a finite family of points in $\Rea^d$ colored with $r$ colors, of size $|P|\geq \max\{m+d,r\}$.
  
  Assume that the convex hull of the union of every $r-k+1$ color classes contains the origin. Then, there is some $\{u_1,\ldots,u_d,v_1,\ldots,v_m\}\subset P$, colored with at least $k$ different colors, such that $0\in \conv(\{u_1,\ldots,u_d,v_i\})$ for all $i\in[m]$.
\end{theorem}
\begin{proof}
For each $p\in P$, let $H_p\subset \Rea^d$ be defined as
\[
    H_p=\{ x\in \Rea^d :\, x\cdot p>0\}.
\]
Note that, for $p\neq 0$, $H_p$ is an open half-space, and for $p=0$, $H_p=\emptyset$. In particular, $H_p$ is convex for all $p\in P$.

It is a well known fact that for any $P'\subset P$, $0\in\conv(P')$ if and only if $\cap_{p\in P'} H_p=\emptyset$. Indeed, assume first that  $\cap_{p\in P'} H_p\neq \emptyset$. That is, there is some $x\in \Rea^d$ such that $x\cdot p>0$ for all $p\in P'$. Let $\{t_p\}_{p\in P'}$ be non-negative real numbers such that $\sum_{p\in P'} t_p=1$. Then, there is some $p_0\in P'$ such that $t_{p_0}>0$. Thus, we obtain
\[
    \left(\sum_{p\in P'} t_p p\right) \cdot x = \sum_{p\in P'} t_p (p\cdot x)\geq t_{p_0}(p_0\cdot x) >0 .
\]
Therefore, $0\notin\conv(P')$. In the other direction, assume that  $0\notin\conv(P')$. Then, there is some hyperplane separating $0$ from $ \conv(P')$. In other words, there exists $x\in \Rea^d$ such that $x\cdot y >0$ for all $y\in \conv(P')$. In particular, $x\in H_p$ for all $p\in P'$.
Therefore, $\cap_{p\in P'} H_p\neq \emptyset$.

We color the family $\mathcal{H}=\{H_p\}_{p\in P}$ in $r$ colors, each set $H_p$ colored with the color of $p$. 
Assume for contradiction that for all $\{u_1,\ldots,u_d,v_1,\ldots,v_m\}\subset P$, colored with at least $k$ colors, $0\notin \conv(\{u_1,\ldots,u_d,v_i\})$ for some $i\in[m]$. Equivalently, for all $\{H_{u_1},\ldots,H_{u_d},H_{v_1},\ldots,H_{v_m}\}\subset\mathcal{H}$, colored with at least $k$ colors, $\{ H_{u_1},\ldots,H_{u_d},H_{v_i}\}$ is intersecting for some $i\in[m]$.

By Theorem \ref{thm:main_geometric}, there are $r-k+1$ color classes in $\mathcal{H}$ whose union is intersecting. Equivalently, there are $r-k+1$ color classes in $P$ that do not contain the origin in the convex hull of their union. But this is a contradiction to the assumption of the theorem.
\end{proof}

For the next section, we will need the special case $r=k=d+m$ of Theorem \ref{thm:caratheodory}:

\begin{theorem}\label{thm:caratheodory_special}
  Let $d\geq 1$ and $m\geq 1$.
  Let $P$ be a finite family of points in $\Rea^d$ colored with $d+m$ colors. Assume that the convex hull of every color class contains the origin. Then, there is some colorful set $\{u_1,\ldots,u_d,v_1,\ldots,v_m\}\subset P$ such that $0\in \conv(\{u_1,\ldots,u_d,v_i\})$ for all $i\in[m]$.
\end{theorem}

\subsection{Tverberg centers for large point sets}

In this section we prove Theorem \ref{thm:tverberg_sunflower}. Theorem \ref{thm:tverberg_sunflower} follows from Theorem \ref{thm:caratheodory_special} by a straightforward adaptation of Sarkaria's proof of Tverberg's Theorem via the Colorful Carath{\'e}odory Theorem \cite{sarkaria1992tverberg}. For completeness, we write down the argument, closely following the exposition in \cite{arocha2009very} (which in turn is based on B\'ar\'any and Onn's variant of Sarkaria's proof given in \cite{barany1997colourful}).

First, we will need the following Lemma:
\begin{lemma}[Sarkaria (see e.g. {\cite[Lemma 2]{arocha2009very}})]
\label{lemma:tensor}
    Let $P_1,\ldots,P_r$ be finite sets of points in $\Rea^d$.  Let $v_1,\ldots,v_r\in \Rea^{r-1}$ be the vertices of the standard $(r-1)$-dimensional simplex. Define
\[
    \bar{P}= \left\{ \begin{pmatrix} x\\1 \end{pmatrix}\otimes v_j :\, j\in[r],\, x\in P_j\right\}\subset \Rea^{(r-1)(d+1)},
\]
where $\otimes$ denotes the Kronecker product. Then $\cap_{j=1}^r \conv(P_j)=\emptyset$ if and only if $0\notin \conv(\bar{P})$.
\end{lemma}

\begingroup
\def\thetheorem{\ref{thm:tverberg_sunflower}}
\begin{theorem}
Let $d\geq 1$, $r\geq 2$, and let $n=(r-1)(d+1)$.
Let $A\subset \Rea^d$ be a finite set of points of size larger than $n$. Then, there exists a partition $A_1,\ldots,A_r$ of $A$ that has a Tverberg center of size $n$.
\end{theorem}
\addtocounter{theorem}{-1}
\endgroup
\begin{proof}
Let $n=(r-1)(d+1)$ and  $m=|A|-n$.
Write $A=\{x_1,\ldots,x_{n+m}\}$. For $1\leq i\leq m+n$ and $1\leq j\leq r$, let $x_i^j$ be a copy of the point $x_i$. For $1\leq j\leq r$, let $P_j=\{x_1^j,\ldots,x_{n+m}^j\}$, and let $P=P_1\cup\cdots\cup P_r$. Let $v_1,\ldots,v_r\in \Rea^{r-1}$ be the vertices of the standard simplex, and define for each $1\leq i\leq n+m$ and $1\leq j\leq r$,
\[
    \bar{x}_i^j = \begin{pmatrix} x_i^j\\1 \end{pmatrix}\otimes v_j\in \Rea^n.
\]
Let $\bar{P}=\{\bar{x}_i^j :\, 1\leq i\leq n+m, 1\leq j\leq r\}$.
We color the points in $P$ with $|A|=n+m$ colors, where, for $1\leq i\leq n+m$ and $1\leq j\leq r$, $x_i^j$ is colored with color $i$. Similarly, we color $\bar{P}$ with $n+m$ colors, where,  for $1\leq i\leq n+m$ and $1\leq j\leq r$, $\bar{x}_i^j$ is colored with color $i$.

Note that for each color class $C_i=\{x_i^j :\, 1\leq j\leq r\}$,  $1\leq i\leq n+m$, we have $\cap_{j=1}^r \conv(C_i\cap P_j)=\cap_{j=1}^r\conv(\{x_i^j\})= \{x_i\}\neq \emptyset$. So, by Lemma \ref{lemma:tensor}, we have
\[
0 \in \conv(\{\bar{x}_i^j :\, 1\leq j\leq r\})
\]
for all $1\leq i\leq n+m$. That is, the convex hull of each of the color classes of $\bar{P}$ contains the origin. Therefore, by Theorem \ref{thm:caratheodory_special}, there is a colorful set $U=\{\bar{x}_{a_1}^{b_1},\ldots,\bar{x}_{a_{n+m}}^{b_{n+m}}\}\subset\bar{P}$ such that 
\begin{equation}\label{eq:conv1}
   0\in \conv(\{\bar{x}_{a_1}^{b_1},\ldots,\bar{x}_{a_n}^{b_n},\bar{x}_{a_{n+i}}^{b_{n+i}}\}) 
\end{equation} 
for all $1\leq i\leq m$. Since $U$ is colorful, the index set $\{a_1,\ldots,a_{n+m}\}$ is a permutation of $[n+m]$. Therefore, we can define a partition $A=A_1\cup\cdots\cup A_r$ by
\[
    A_j= \{x_{a_i} :  \, b_i=j\}
\]
for $1\leq j\leq r$. 
Let $B=\{ x_{a_1},\ldots,x_{a_n}\}\subset A$.
By Lemma \ref{lemma:tensor},  \eqref{eq:conv1} implies that for all $1\leq i\leq m$,
\[
\bigcap_{j=1}^r \conv\left(\{x_{a_1}^{b_1},\ldots,x_{a_n}^{b_n},x_{a_{n+i}}^{b_{n+i}}\}\cap P_j\right)
=
    \bigcap_{j=1}^r \conv\left((B\cup \{x_{a_{n+i}}\})\cap A_j\right)\neq \emptyset,
\]
as wanted.
\end{proof}

\begin{remark}
    For any point set in $\Rea^d$ of size larger than $(r-1)(d+1)$, Theorem \ref{thm:tverberg_sunflower} guarantees the existence of a partition into $r$ parts with Tverberg center of size at most $(r-1)(d+1)$. Note, however, that the same point set may have other Tverberg partitions without such a center (see Figure~\ref{fig:tverberg_example}).
\end{remark}
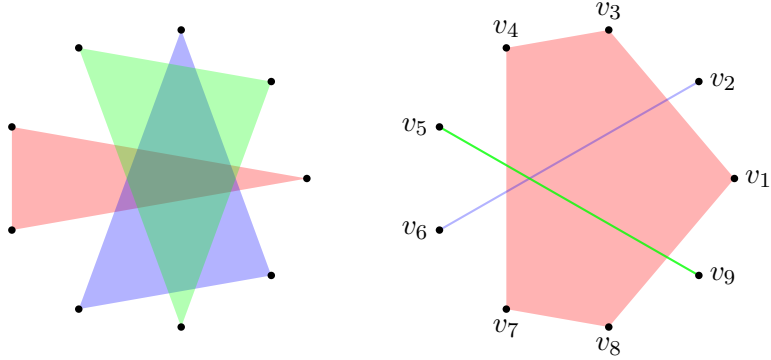
\begin{figure}
\begin{center}
\begin{tikzpicture}[thick, scale=1]
\begin{scope}[xshift=-80, yshift=0]
\fill[red, opacity=0.3] (0:2)--(160:2)--(200:2)--(0:2);
\fill[blue, opacity=0.3] (80:2)--(240:2)--(320:2)--(80:2);
\fill[green, opacity=0.3] (40:2)--(120:2)--(280:2)--(40:2);

\filldraw (0:2) circle(1pt);
\filldraw (40:2) circle(1pt);
\filldraw (80:2) circle(1pt);
\filldraw (120:2) circle(1pt);
\filldraw (160:2) circle(1pt);
\filldraw (200:2) circle(1pt);
\filldraw (240:2) circle(1pt);
\filldraw (280:2) circle(1pt);
\filldraw (320:2) circle(1pt);
\end{scope}

\begin{scope}[xshift=80, yshift=0]
\fill[red, opacity=0.3] (0:2)--(80:2)--(120:2)--(240:2)--(280:2);
\draw[blue, opacity=0.3] (40:2)--(200:2);
\draw[green, opacity=0.8] (160:2)--(320:2);

\node[right] at (0:2){$v_1$};
\node[right] at (40:2){$v_2$};
\node[above] at (80:2){$v_3$};
\node[above] at (120:2){$v_4$};
\node[left] at (160:2){$v_5$};
\node[left] at (200:2){$v_6$};
\node[below] at (240:2){$v_7$};
\node[below] at (280:2){$v_8$};
\node[right] at (320:2){$v_9$};

\filldraw (0:2) circle(1pt);
\filldraw (40:2) circle(1pt);
\filldraw (80:2) circle(1pt);
\filldraw (120:2) circle(1pt);
\filldraw (160:2) circle(1pt);
\filldraw (200:2) circle(1pt);
\filldraw (240:2) circle(1pt);
\filldraw (280:2) circle(1pt);
\filldraw (320:2) circle(1pt);
\end{scope}
\end{tikzpicture}
\caption{Two Tverberg partitions of $9$ points in the plane into $3$ parts.
The partition on the left does not have a Tverberg center of size $6$, while the partition on the right has one: $\{v_2,v_4,v_5,v_6,v_7,v_9\}$.}
\label{fig:tverberg_example}
\end{center}
\end{figure}    

Let $A\subset \Rea^d$ be a finite set of points. We say that a partition $A_1,\ldots,A_r$ of $A$ is a \emph{Tverberg partition with tolerance $t$} if 
\[
    \bigcap_{i=1}^r \conv(A_i\setminus C)\neq\emptyset,
\]
for any $C\subset A$ of size at most $t$. In other words, the partition has tolerance $t$ if it remains a Tverberg partition even after removing any $t$ points from $A$.
Larman showed in \cite{larman1972sets} that any $2d+3$ points in $\Rea^d$ have a Tverberg partition into two parts (that is, a Radon partition)  with tolerance $1$. This was later extended by Garc{\'\i}a-Col{\'\i}n (\cite{colin2007applying,garcia2015projective}) who showed that any $(t+1)(d+1)+1$ points in $\Rea^d$ have a Radon partition with tolerance $t$, and by Sober{\'o}n and Strausz in \cite{soberon2012generalisation}, who showed that any $(t+1)(r-1)(d+1)+1$ points in $\Rea^d$ have a Tverberg partition into $r$ parts with tolerance $t$. Note that this bound is not sharp in general: in \cite{garcia2017note} it was shown that any set of $rt+o(t)$ points in $\Rea^d$ has a Tverberg partition into $r$ parts with tolerance $t$, giving an improved bound for large values of $t$ (see \cite{soberon2018robust} for a further improvement of this bound). In \cite{mulzer2014algorithms}, an improved bound was given in the case $d=1$, and in the case $d=2$ for some values of $r,t$.

It was brought to our attention by Andreas Holmsen that one can recover Sober{\'o}n and Strausz's result as a consequence of Theorem \ref{thm:tverberg_sunflower}, in the following way: 

\begingroup
\def\thetheorem{\ref{thm:soberon_strausz}}
\begin{theorem}[Sober{\'o}n-Strausz {\cite[Theorem 1]{soberon2012generalisation}}]
Let $d\geq 1$, $r\geq 2$, and $t\geq 0$. Let $A\subset \Rea^d$ be finite set of points of size at least $(t+1)(r-1)(d+1)+1$. Then, $A$ has a Tverberg partition into $r$ parts with tolerance $t$.
\end{theorem}
\addtocounter{theorem}{-1}
\endgroup
\begin{proof}
    We argue by induction on $t$. For $t=0$, the claim is just Tverberg's theorem.
    Assume $t\geq 1$. By Theorem \ref{thm:tverberg_sunflower}, there exists a partition $A_1,\ldots,A_r$ of $A$ and a set $B\subset A$ 
 of size $(r-1)(d+1)$ such that
    \[
        \bigcap_{i=1}^r \conv((B\cup\{p\})\cap A_i)\neq\emptyset
    \]
    for any $p\in A\setminus B$.
    Note that $|A\setminus B|\geq t(r-1)(d+1)+1$. Therefore, by the induction hypothesis, there is a Tverberg partition $B_1,\ldots,B_r$ of $A\setminus B$ with tolerance $t-1$.

    We denote by $S_r$ the set of permutations of $[r]$.
    Since $|A\setminus B|> t(r-1)(d+1)\geq tr$, we have
\begin{align*}
\frac{1}{r!}\sum_{\pi\in S_r} |\cup_{i=1}^r A_i\cap B_{\pi(i)}|
&=\frac{1}{r!}\sum_{\pi\in S_r} \sum_{i=1}^r |A_i\cap B_{\pi(i)}|
=\frac{1}{r!} \sum_{i=1}^r  \sum_{\pi\in S_r} |A_i\cap B_{\pi(i)}|
\\
&=\frac{1}{r!} \sum_{i=1}^r  (r-1)! |A_i\setminus B| 
= \frac{|A\setminus B|}{r}>t.
\end{align*}
Therefore, there is some $\pi\in S_r$ such that $|\cup_{i=1}^r A_i\cap B_{\pi(i)}|\geq t+1$. Let $D=\cup_{i=1}^r A_i\cap B_{\pi(i)}$. We define a new partition $\tilde{A}_1,\ldots,\tilde{A}_r$ of $A$ by
\[
    \tilde{A}_i = (A_i\cap B)\cup (B_{\pi(i)}\cap(A\setminus B))
\]
for all $1\leq i\leq r$. Note that $\tilde{A}_i\cap D=A_i\cap D$ for all $i$.

We are left to show that $\tilde{A}_1,\ldots,\tilde{A}_r$ is a Tverberg partition of $A$ with tolerance $t$. Let $C\subset A$ be a subset of size at most $t$. We divide into two cases. First, assume that $C$ is disjoint from $B$. Then, since $|C|\leq t$ and $|D|\geq t+1$, there is some $p\in D\setminus C$. We obtain
\begin{multline*}
\bigcap_{i=1}^r \conv(\tilde{A_i}\setminus C) \supset  \bigcap_{i=1}^r \conv((\tilde{A_i}\cap (B\cup D))\setminus C) 
\\
= \bigcap_{i=1}^r \conv((A_i\cap(B\cup D) )\setminus C) \supset \bigcap_{i=1}^r \conv(A_i\cap(B\cup \{p\}))\neq \emptyset,
\end{multline*}
as desired. Next, assume that $C\cap B\neq \emptyset$. Let $C'=C\cap (A\setminus B)$. Then $|C'|\leq t-1$, and therefore, since $B_1,\ldots,B_r$ is a Tverberg partition of $A\setminus B$ with tolerance $t-1$, we obtain
\begin{multline*}
\bigcap_{i=1}^r \conv(\tilde{A_i}\setminus C) \supset  
\bigcap_{i=1}^r\conv((\tilde{A_i}\cap (A\setminus B))\setminus C')
\\
=\bigcap_{i=1}^r\conv( B_{\pi(i)}\setminus C')
=\bigcap_{i=1}^r\conv( B_{i}\setminus C') \neq \emptyset.
\end{multline*}
Thus, $\tilde{A}_1,\ldots,\tilde{A}_r$ is a Tverberg partition of $A$ with tolerance $t$, as wanted.
\end{proof}

\section{Concluding remarks}\label{sec:conclusion}

It would be interesting to prove a stronger version of Theorem \ref{thm:very_colorful_helly_d_leray}, analogous to Theorem \ref{thm:very_colorful_helly_d_coll}:

\begin{conjecture}
Let $d\geq 1$, $r\geq d+1$, $1\leq m\leq r$ and $m\leq k\leq \min\{m+d,r\}$ be integers. Let $V$ be a finite set with $|V|\geq \max\{m+d,r\}$.
Let $X$ be a $d$-Leray simplicial complex on vertex set $V$, and let $M$ be a matroid of rank $r$ on vertex set $V$ with rank function $\rho$.

Assume that for every $U=\{u_1,\ldots,u_d,v_1,\ldots,v_m\}\subset V$ with $\rho(U)\geq k$, there is some $i\in[m]$ such that $\{u_1,\ldots,u_d,v_i\}\in X$. 
Then, there is some $\tau\in X$ such that  $\rho(V\setminus \tau)\leq k-1$.
\end{conjecture}

Another interesting question, suggested to us by Florian Frick, is whether a version of Theorem \ref{thm:tverberg_sunflower} holds in the context of the topological Tverberg theorem. That is,

\begin{conjecture}
    Let $d\geq 1$, $r\geq 2$ a prime power, and $n=(r-1)(d+1)$. Let $N>n$, and let $\Delta_N$ be the $N$-dimensional simplex. Then, for every continuous map
    $
      f: \Delta_N\to \Rea^d,
    $
    there exist $r$ pairwise disjoint faces $F_1,\ldots,F_r$ of $\Delta_N$, and an
     $(n-1)$-dimensional face $F$ of $\Delta_N$, such that for each $n$-dimensional face $F'$ of $\Delta_N$ containing $F$,
     $\cap_{i=1}^r f(F_i\cap F')\neq\emptyset$.
\end{conjecture}

\section*{Acknowledgements}
We thank Florian Frick, Andreas Holmsen and Roy Meshulam for their helpful comments and suggestions. We thank Andreas Holmsen for sharing with us his argument for deriving Theorem \ref{thm:soberon_strausz} from Theorem \ref{thm:tverberg_sunflower}.
 
\bibliographystyle{abbrv}
\bibliography{biblio}

\end{document}